\documentclass{amsart}
\usepackage{verbatim,amssymb,amsmath,amscd,latexsym,amsbsy,mathrsfs,mathtools}
\usepackage{enumerate}
 \input{xy}
\xyoption{all}

\newtheorem{thm}{Theorem}[section] \newtheorem{pro}[thm]{Proposition}
\newtheorem{lemma}[thm]{Lemma}
\newtheorem{cor}[thm]{Corollary}
\numberwithin{equation}{section}

\theoremstyle{remark}

\theoremstyle{definition} 
\newtheorem{rmk}[thm]{Remark} \newtheorem*{clm}{Claim}
\newtheorem{remark}[thm]{Remark}

\newtheorem{definition}[thm]{Definition}

\DeclareMathAlphabet{\mathpzc}{OT1}{pzc}{m}{it}

 \DeclareMathOperator*{\QF}{QF}
\DeclareMathOperator*{\Gal}{Gal} 
 
\DeclareMathOperator*{\degree}{deg}

\newcommand{\RR}{\mathbb{R}} 
\newcommand{\ZZ}{\mathbb{Z}} 
\newcommand{\PP}{\mathbb{P}} \newcommand{\FF}{\mathbb{F}}

\newcommand{\cO}{\mathcal{O}}

\newcommand{\Z}{\mathbb{Z}}

\begin{document}
\title{On the compositum of wildly ramified extensions} 
 \author{ Manish Kumar
  }
 \address{ 
Statistics and Mathematics Unit\\
Indian Statistical Institute, \\
Bangalore, India-560059
  }
  \email{manish@isibang.ac.in}
 \begin{abstract}
  We compute the ramification filtration on wildly ramified $p^2$-cyclic extensions of local fields of characteristic $p$. The ramification filtration on the compositum of two $p$-cyclic and $p^2$-cyclic extensions are also computed. As an application, some partial results towards Abhyankar's Inertia conjecture has been proved. 
 \end{abstract}
\maketitle

\section{Introduction}
 The higher ramification filtration of a wildly ramified extension contains vital information about the extension. For instance the degree of the different is encoded in this data (Hilbert's different formula). This in turn helps in computing genus of a (wildly ramified) cover of a given curve. But computing these filtrations is a difficult task even for cyclic Artin-Schrier-Witt extensions. For a cyclic extension of degree $p$, the computation is not hard and can be found at many places (for instance see \cite[IV, 2, Exercise 5]{Serre-local.fields}). In this manuscript, we compute the ramification filtration for $p^2$-cyclic extension (Lemma \ref{p^2-cyclic-filtration}). 
 We also compute the ramification filtration for the compositum of two $p$-cyclic and $p^2$-cyclic extensions in a large number of cases (Proposition \ref{p-cyclic-compositum}, \ref{filtration.linearly.disjoint}, \ref{filtration.not.linearly.disjoint}). The case of $p^3$-cyclic extension already runs into computationally intractable problem. These computations involve working with the operations in the Witt ring.
 
 For an algebraically closed field $k$ of characteristic $p$ in \cite{Harbater-moduli.of.pcovers}, Harbater showed that every Galois cover of the local field $k((x))$ with Galois group a $p$-group $P$ can be extended to a $P$-cover of $\PP^1_k$ branched only at $x=0$. This was extended by Katz to all finite Galois covers of $k((x))$ in \cite{Katz-Gabber}. A local $p^2$-cyclic extension of $k((x))$ is given by a Witt vector $(a_0,a_1)\in W_2(k((x)))$ of length two. And a $(\Z/p\Z)^2$-extension of $k((x))$ is given by the compositum of two distinct Artin-Schrier extension of $k((x))$ corresponding to elements $b_1,b_2\in k((x))$. As an application of our local calculations, we compute the genus of any Harbater-Katz-Gabber cover of $\PP^1$ associated to Galois extensions of degree $p^2$ in terms of valuations of $a_1,a_2,b_1,b_2$ (Theorem \ref{genus-p^2-noncyclic-covers}, \ref{genus-p^2-cyclic-covers}).  
 
 Let $G$ be a quasi-$p$ group, i.e. a group generated by its Sylow-$p$ subgroup, and $I\le G$ be such that $I=\Z/n\Z \rtimes P$ where $P$ is a $p$-group whose conjugates generate $G$ and $(n,p)=1$. Then Abhyankar's Inertia conjecture asserts that there exist a $G$-Galois cover $X\to \PP^1_k$ branched only at $\infty$ such that the inertia group at a point of $X$ lying above $\infty$ is $I$. It is easy to see that inertia group at any ramified point of a Galois cover of $\PP^1_k$ branched only at $\infty$ have the above mentioned property. For a pair $(G,I)$ as above, we will say $(G,I)$ is realizable if Abhyankar's Inertia conjecture is true for $(G,I)$. This conjecture is largely open though there are some results in support of the conjecture. For instance Harbater in \cite{Harbater-moduli.of.pcovers} showed that if $(G,P)$ is realizable for a $p$-subgroup $P$ of $G$ and $Q$ is a $p$-subgroup of $G$ containing $P$ then $(G,Q)$ is realizable. There are some results for specific Galois groups, for instance, see \cite{BP}, \cite{MP} and \cite{Obus}.
 
 As an application to the computation of the ramification filtration of the compositum of local extensions, it is shown in Theorem \ref{p-cyclic.main} that if $(G,I)$ is realizable and $P$ is a $p$-subgroup of $G$ containing $I_2$ then $(G\times \Z/p\Z,Q)$ is also realizable for some index $p$-subgroup $Q$ of $P\times \Z/p\Z$ under the assumption that there is no epimorphism from $G$ to $\Z/p\Z$. Here $I_2$ appears in the lower numbering ramification filtration of $I$.
 Moreover if $(G,I)$ is realizable and $P$ is any $p$-subgroup of $G$ containing $I_2$ with an epimorphism $a:P\to \Z/p^2\Z$  then $(G\times\Z/p^2\Z, Q)$ is realizable where $Q$ is an index $p$ subgroup of $P\times_{\Z/p\Z}\Z/p^2\Z$ (Theorem \ref{p^2-cyclic.main}). One of the ingredient in the proof is the main result of \cite{wild.ram}.

\subsection*{Acknowledgments}
 A part of this work was done while the author was at Universit\"at Duisburg-Essen, where he was supported by SFB/TR-45 grant.

\section{Ramification filtration and Artin-Schrier-Witt theory}

 Let $L/K$ be a Galois extension of local fields with Galois group $G$. Let $v_K$ and $v_L$  denote the valuation associated to $K$ and $L$ respectively, with the value group $\ZZ$. Let us define a decreasing filtration on $G$ by 
 $$G_i=\{\sigma\in G: v_S(\sigma x -x)\ge i+1, \, \forall x\in S\}$$
 Note that $G_{-1}=G$ and $G_0$ is the inertia subgroup. This filtration is called the lower (numbering) ramification filtration. For every $i$, $G_i$ is a normal subgroup of $G$. One extends this filtration to the real line as follows: for $u\in \RR, u\ge -1$, let $m$ be the smallest integer such that $m\ge u$ then $G_u=G_m$. Then the upper (numbering) ramification filtration on $G$ is defined as $G^v=G_{\psi(v)}$ where $\psi$ is the inverse of the Herbrand function $\phi$ given by
 $$\phi(v)=\int_{0}{^v}\frac{du}{[G:G_u]}$$
 Note that $\phi$ is bijective piece-wise linear function. Let $G^{v+}=\cup_{\epsilon>0} G^{v+\epsilon}$.

 A number $u\ge 0$ (respectively $l \ge 0$) is called an upper jump (respectively lower jump) of the ramification filtration of $G$ if $G^u\ne G^{u+}$ (respectively $G_l\ne G_{l+}$. Let $u_1,\ldots, u_r$ be the upper jumps, $l_1,\ldots, l_r$ be the lower jumps and $s_i=[G:G^{u_i}]=[G:G_{l_i}]$. Note that $G_0$ is a $p$-group iff $s_1=1$, i.e. $L/K$ is purely wildly ramified. 
 \begin{remark}\label{upper-lower-conversion}
  A straight forward computation shows that if $G_0$ is a $p$-group and if we set $l_0=u_0=0$ then for $i\ge 1$, $$u_i=\sum_{j=1}^i\frac{l_j-l_{j-1}}{s_j}\text{\ \ \ \ \  and \ \ \ \ \ }l_i=\sum_{j=1}^i(u_j-u_{j-1})s_j$$.
 \end{remark}

 \begin{rmk}
  Note that $G_i=G$ for some $i\ge 0$ iff $G^i=G$. Since $\phi(v)\le v$, $G_v=G^{\phi(v)}\supset G^v$. This explains the ``if part''. The ``only if'' is true because for $v\le i$, $\phi(v)=v$ and hence $\psi(v)=v$.
 \end{rmk}

 \begin{lemma}\label{upperjumpslift}
  Let $L/K$ be a finite Galois extension of local fields with Galois group $G$ and $H$ be a normal subgroup of $G$. 
  \begin{enumerate}
   \item Let $u_1,\ldots, u_r$ be the upper jumps of $G/H$ then $u_1,\ldots, u_r$ are also upper jumps of $G$.
   \item If $H=G_i$ for some $i$, and $l_1,\ldots,l_r$ are the lower jumps of $H$ then $l_1,\ldots,l_r$ are also lower jumps of $G$.
  \end{enumerate}
 \end{lemma}

 \begin{proof}
  The statement (1) follows immediately from \cite[IV, 1, Proposition 14]{Serre-local.fields} by noting that if $G^{u+}H/H \subsetneq G^uH/H$ then $G^{u+}\subsetneq G^u$. The statement (2) is a consequence of \cite[IV, 1, Proposition 2]{Serre-local.fields}.
 \end{proof}

 \begin{cor}\label{cor.to.upperjumpslift}
  Let $L/K$ and $M/K$ be Galois extensions of local fields with the upper jumps of their Galois groups $U=\{u_1,\ldots, u_n\}$ and $V=\{v_1,\ldots v_m\}$ respectively. Let $N$ be the cardinality of the set $U\cup V$. If $[LM:K]=p^N$ then $U\cup V$ is the set of all the upper jumps of $\Gal(LM/K)$.
 \end{cor}

 \begin{proof}
  Let $G=\Gal(LM/K)$. Applying the above lemma with $H$ as $\Gal(LM/L)$ and $\Gal(LM/M)$ we see that all the elements of $U\cup V$ are upper jumps of $G$. Since $[LM:K]=p^N$, $G$ can have at most $N$ upper jumps.
 \end{proof}

 Let $K$ be field of characteristic $p>0$ and $\bar K$ its algebraic closure. Let $W_n(K)$ denote the Witt ring of length $n$. The ring operations of Witt rings are a little complicated. For instance, let $(a_0,a_1), (b_0,b_1)\in W_2(K)$ then:

 \begin{equation*}
  \begin{split}
   (a_0,a_1)+_w(b_0,b_1)=& \big(a_0+b_0,\ a_1+b_1-\big[a_0^{p-1}b_0+\frac{(p-1)}{1}a_0^{p-2}b_0^2 \\ 
   &+\frac{(p-1)(p-2)}{2}a_0^{p-3}b_0^3+ \ldots+\frac{(p-1)!}{(p-1)!}a_0b_0^{p-1}\big]\big)
  \end{split}
 \end{equation*}
 \begin{equation}\label{witt.substraction}
  \begin{split}
   (a_0,a_1)-_w(b_0,b_1)=\big( a_0-b_0,\ &a_1-b_1-\big[-a_0^{p-1}b_0+\frac{(p-1)}{1}a_0^{p-2}b_0^2-\ldots\\ 
   &+\frac{(p-1)!}{(p-1)!}a_0b_0^{p-1}\big]\big) 
  \end{split}
 \end{equation}

 \begin{definition}
  Let $L/K$ be a field extension. We say an element $\alpha\in L\setminus K$ is an \emph{AS-element} of $L/K$ if $\alpha^p-\alpha\in K$. Moreover if $(K,v_K)$ is a local field, $L/K$ is totally ramified and $v_K(\alpha^p-\alpha)$ is coprime to $p$ then $\alpha$ is said to be a \emph{reduced $AS$-element} and $\alpha^p-\alpha$ to be a reduced element of $K$.
 \end{definition}

 \begin{definition}
  Let $(a_0,\ldots, a_{n-1})\in W_n(K)$. We will say that $L/K$ is a field extension corresponding to $(a_0,\ldots, a_{n-1})$ if there exists $(\alpha_0,\ldots, \alpha_{n-1})\in W_n(L)$ such that $(\alpha_0^p,\ldots, \alpha_{n-1}^p)-_w(\alpha_0,\ldots, \alpha_{n-1})=(a_0,\ldots,a_{n-1})$ and $L=K(\alpha_0,\ldots, \alpha_{n-1})$.
 \end{definition}

 \begin{remark} \label{reduced}
  Let $L/K$ be a totally ramified extension of complete local fields with perfect residue field. If $\alpha$ is an AS-element of $L/K$ then there exist $x\in K$ such that $\alpha-x$ is a reduced AS-element of $L/K$. In fact more generally, for any $(a_0,\ldots, a_{n-1})\in W_n(K)$ there exist $(a'_0,\ldots, a'_{n-1})\in W_n(K)$ such that $v_K(a'_i)$ is coprime to $p$ for all $i$ and $(a_0,\ldots, a_{n-1})-_w (a'_0,\ldots, a'_{n-1})=(x_0^p,\ldots, x_{n-1}^p)-_w(x_0,\ldots,x_{n-1})$ for some $(x_0,\ldots,x_{n-1})\in W_n(K)$ (\cite{schmid}, \cite[Proposition 4.1]{thomas}). We shall say $(a'_0,\ldots,a'_{n-1})$ is \emph{reduced}.
 \end{remark}

 \begin{definition}
  Let $L/K$ be a compositum of Artin-Schrier extensions. A subset $\{\alpha_1,\ldots,\alpha_n\} \in L\setminus K$ is an \emph{AS-generating set} of $L/K$ if $\alpha_1,\ldots,\alpha_n$ are AS-elements and $L=k(\alpha_1,\ldots,\alpha_n)$. Moreover the above set will be called an \emph{AS-basis} if $[L:K]=p^n$.
 \end{definition}

 \begin{lemma}\label{AS-basis}
  Let $L=K(\alpha_1,\ldots,\alpha_n)$ be a compositum of Artin-Schrier extensions, where $\alpha_i\in \bar K$, $f_i=\alpha_i^p-\alpha_i\in K$ for all $i$ and $\{\alpha_1,\ldots,\alpha_n\}$ is an AS-basis of $L/K$. Let $\gamma\in L$ be such that $\gamma^p-\gamma\in K$ then $\gamma=a_1\alpha_1+\ldots a_n\alpha_n+x$ for some $a_1,\ldots, a_n\in \FF_p$ and $x\in K$.   
 \end{lemma}

 \begin{proof}
  Note that $G=\Gal(L/K)=(\ZZ/p\ZZ)^n$. By Galois theory $\Gal(L/K(\gamma))$ is an index $p$ subgroup of $G$. Let $\gamma(a_1,\ldots, a_n)=a_1\alpha_1+\ldots +a_n\alpha_n\in L$ where $a_1,\ldots, a_n\in \FF_p$. Then $\gamma(a_1,\ldots, a_n)^p-\gamma(a_1,\ldots, a_n)\in K$. Moreover, as $1, \alpha_1,\ldots \alpha_n$ are linearly independent over $K$, some $a_i\ne 0$. So $\gamma(a_1,\ldots, a_n)\in L\setminus K$. Hence $\Gal(L/K(\gamma(a_1,\ldots, a_n)))$ is an index $p$-subgroup of $G$ if $(a_1,\ldots,a_n)\ne (0,\ldots,0)$. From Artin-Schrier theory, we know that for $b_1,\ldots, b_n\in \FF_p$, $K(\gamma(a_1,\ldots, a_n))=K(\gamma(b_1,\ldots, b_n))$ iff 
  $$\gamma(a_1,\ldots, a_n)^p-\gamma(a_1,\ldots, a_n)=c(\gamma(b_1,\ldots, b_n)^p-\gamma(b_1,\ldots, b_n))+x^p-x$$ for some $c\in \FF_p$ and $x\in K$. Equivalently, $a_1f_1+\ldots+a_nf_n=c(b_1f_1+\ldots+b_nf_n)+x^p-x$. 
  This is equivalent to $x^p-x=(a_1-cb_1)f_1+\ldots+(a_n-cb_n)f_n$. Since for each $i$, $K(\alpha_i)$ and $K(\alpha_j| 1\le j\le n, j\ne i)$ are linearly disjoint over $K$, no nontrivial $\FF_p$-linear combination of $f_1,\ldots, f_n$ is of the form $x^p-x$. Hence $x^p-x=(a_1-cb_1)f_1+\ldots+(a_n-cb_n)f_n$ is equivalent to $(a_1,\ldots, a_n)=c(b_1,\ldots, b_n)$. 
 
  So we have shown that there are $(p^n-1)/(p-1)$ many distinct $p$-cyclic intermediate extensions of the form $K(\gamma(a_1,\ldots,a_n))/K$ of $L/K$. Also there are $(p^n-1)/(p-1)$ subgroups of index $p$ in $G$. Hence by Galois theory $K(\gamma)=K(\gamma(a_1,\ldots,a_n))$ for some $a_1,\ldots,a_n$ not all zero. But this implies $\gamma^p-\gamma=c(\gamma(a_1,\ldots,a_n)^p-\gamma(a_1,\ldots,a_n))+x^p-x$ for some $c\in \FF_p$ and $x\in K$. Simplifying, we obtain $$(\gamma-(ca_1\alpha_1+\ldots+ca_n\alpha_n)-x)^p-(\gamma-(ca_1\alpha_1+\ldots+ca_n\alpha_n)-x)=0$$
  Hence $\gamma=(ca_1\alpha_1+\ldots+ca_n\alpha_n)+x+d$ for some $d\in \FF_p$.
 \end{proof}

 \begin{pro}
  Let $L/K$ and $M/K$ be linearly disjoint Galois extensions with Galois groups $G_1$ and $G_2$ respectively. Let $v$, $w_1$ and $w_2$ be the valuation associated to $K$, $L$ and $M$ respectively with $w_1$ and $w_2$ being extensions of $v$. Let $I_1$ and $I_2$ be the decomposition subgroup of $G_1$ and $G_2$ at $w_1$ and $w_2$ respectively. The decomposition subgroup of the compositum $LM/K$ at a valuation $w$ of $LM$ which extends $w_1$ and $w_2$ is the subgroup $I_1\times_I I_2$ of $G_1\times G_2$ where $I=\Gal(\hat L \cap \hat M/\hat K)$.
 \end{pro}

 \begin{proof}
  Note that all fields in the above proposition can be viewed as subfields of $\widehat{LM}$, the completion of $LM$ along $w$. Moreover, the inclusion of $\hat L \cap \hat M$ in $\hat L$ (respectively $\hat M$) gives an epimorphism $I_1\to I$ (respectively $I_2\to I$).

  By \cite[Lemma 3.1]{wild.ram} the decomposition subgroup of $LM/K$ at $w$ is isomorphic to $\Gal(\hat L \hat M/\hat K)$. But this is same as $I_1\times_I I_2$ since $\Gal(\hat L/\hat K)=I_1$ and $\Gal(\hat M/\hat K)=I_2$. 
 \end{proof}

\section{Local theory}

 Let $K$ be a local field of characteristic $p$, $R$ the associated valuation ring, $v$ the valuation and $k$ the residue field. Note that a $p$-cyclic extension of $K$ has only one jump in the ramification filtration.
 \begin{lemma}\label{AS-upper}
  Let $L/K$ be a Galois extension of local fields. Let $\alpha\in L$ be such that $\alpha^p-\alpha\in K$ and $v(\alpha^p-\alpha)=-i$ is prime to $p$. Then $i$ is an upper jump of $G=\Gal(L/K)$. 
 \end{lemma}
 \begin{proof}
  The inclusion $K(\alpha) \subset L$ induces an epimorphism $G\to G/H=\bar G=\Gal(K(\alpha)/K)$ where $H=\Gal(L/K(\alpha))$. We know that the only upper jump of $\bar G$ is $i$ (\cite[IV, 2, Exercise 5]{Serre-local.fields}). By Lemma \ref{upperjumpslift}, $i$ is an upper jump of $G$ as well.
 \end{proof}

 \begin{lemma}
  Let $L$ and $M$ be totally ramified Galois extensions of $K$ with Galois groups of exponent $p$. Suppose the upper numbering jumps $\{u_1,\ldots,u_l\}$ and $\{v_1,\ldots, v_m\}$ of $L/K$ and $M/K$ respectively are all distinct, i.e., $u_i\ne v_j$ for all $i,j$. Then the first upper jump $w_1$ of $LM/K$ is the minimum of $u_1$ and $v_1$. Moreover $\Gal(LM/K)^{w_1+}=\Gal(L/K)^{w_1+}\times\Gal(M/K)^{w_1+}$.
 \end{lemma}

 \begin{proof}
  Let $G=\Gal(L/K)$, $H=\Gal(M/K)$ and $\Gamma=\Gal(LM/K)$. Without loss of generality suppose $u_1 < v_1$. Since $G$ and $H$ are $p$-groups $G^{u_1}=G=G_{u_1}$ and $H^{v_1}=H=H_{v_1}$. By Lemma \ref{upperjumpslift}, $\Gamma$ has an upper jump at $u_1$. So $w_1 \le u_1$.  
  
  Suppose $w_1 < u_1$, then there exist $\gamma \in LM\setminus K$ such that $\gamma^p-\gamma\in K$ and $v(\gamma^p-\gamma)=-w_1$. Since $w_1$ is also the first lower jump of $\Gamma$, $(p,w_1)=1$.
  By Lemma \ref{AS-basis} and Remark \ref{reduced}, $\gamma=a\alpha+b\beta +x$ for some $\alpha\in L$, $\beta \in M$ reduced AS element w.r.t. $K$, $a,b\in \FF_p$ not both zero and $x\in K$. By Lemma \ref{AS-upper}, $-v(\alpha^p-\alpha)$ and $-v(\beta^p-\beta)$ are equal to one of the upper jumps of $G$ and $H$ respectively. Since the upper numbering jumps of $G$ and $H$ are distinct $v(\alpha^p-\alpha)\ne v(\beta^p-\beta)$. If $v(x)\ge 0$ then $v(x^p-x)\ge 0 > min(v(a(\alpha^p-\alpha)),v(b(\beta^p-\beta)))$ otherwise $v(x^p-x)$ is a multiple of $p$. So $-w_1=v(\gamma^p-\gamma)=\min\{v(a(\alpha^p-\alpha)),v(b(\beta^p-\beta))\}$. But this implies $w_1\ge u_1$, a contradiction!
 \end{proof}

 \begin{pro}\label{p-cyclic-compositum}
  Let $L$ and $M$ be distinct totally ramified Artin-Schrier extensions of $K$ corresponding to $f$ and $g$ respectively, where $f, g \in K$ are reduced. Let $i$ be the upper jump of $L/K$ and $j$ be the upper jump of $M/K$. The ramification filtration on the Galois group $G=\Gal(LM/K)$ is given as follows:
  \begin{enumerate}
   \item If $i < j$ then the upper jumps are $i$ and $j$ with $G^i=G=(\Z/p\Z)^2$ and $G^j=\Z/p\Z$.
   \item $LM/K$ is not totally ramified iff for some $a\in \FF_p$ and $x\in K$, $f+ag+x^p-x\in k$. In this case, $i=j$, the inertia group is $\Z/p\Z$ and the only upper jump is at $i$. 
   \item If $i=j$ and $LM/K$ is totally ramified then there are two cases. If $l=-v(f+ag+x^p-x)< i$ for some $a\in \FF_p$ and $x\in K$ then $l\ge 1$ and the upper jumps are $l$ and $i$ with $G^l=G=(\Z/p\Z)^2$ and $G^i=\Z/p\Z$. Other wise there is only one upper jump at $i$ and $G^i=G=(\Z/p\Z)^2$.
  \end{enumerate}
 \end{pro}
 \begin{proof}
  Let $\alpha,\beta\in \bar K$ be such that $\alpha^p-\alpha=f$ and $\beta^p-\beta=g$. Since the upper jump of $L/K$ is $i$, $v(f)=-i$. Similarly, $v(g)=-j$.
 
  Note that (1) follows from Lemma \ref{AS-upper} and Corollary \ref{cor.to.upperjumpslift}.

  For (2) note that if $f+ag+x^p-x\in k$ for some $a\in \FF_p$ and $x\in K$ then $K(\alpha+a\beta)/K$ is an unramified extension of $K$. Hence $LM/K$ is not totally ramified. Conversely, if $LM/K$ is not totally ramified then there exist $\gamma \in LM \setminus K$ such that $\gamma^p-\gamma \in k$, since $\Gal(LM/K)=(\Z/p\Z)^2$ and $K$ is a local field of characteristic $p$. Now by Lemma \ref{AS-basis}, $\gamma=a\alpha+b\beta +x$ for some $x\in K$ and $a,b\in \FF_p$. So $\gamma^p-\gamma=a(\alpha^p-\alpha)+b(\beta^p-\beta)+x^p-x\in k$. Hence $af+bg+x^p-x\in k$. If $a\ne 0$ then dividing by $a$ we get $f+a'g+x'^p-x'\in k$ for some $a'\in\FF_p$ and $x'\in K$. Otherwise $\gamma=b\beta+x \notin K$, so $b\ne 0$. Hence $v(\gamma^p-\gamma)=v(bg+x^p-x)<0$. This contradicts $\gamma^p-\gamma \in k$. The rest of (2) follows.

  For (3) note that if $i=j$ and $l=-v(f+ag+x^p-x)<i$ for some $a\in \FF_p$ and $x\in K$ then $l\ge 1$. This is because if $l=0$ then $LM/K$ will not be totally ramified. Let $h=f+ag+x^p-x$ and $\gamma$ be such that $\gamma^p-\gamma=h$ then $LM=LK(\gamma)$. So we are reduced to (1). Hence the upper jumps are $l$ and $i$.

  Finally in the remaining scenario, if $\gamma\in LM$ is such that $\gamma^p-\gamma\in K$ and $(v(\gamma^p-\gamma),p)=1$ then in view of Lemma \ref{AS-basis}, $v(\gamma^p-\gamma)=-i$. Hence by \cite[Proposition 2.7]{wild.ram} the first upper jump for $LM/K$ is $i$. But the highest upper jump for $LM/K$ is the maximum of upper jumps for $L/K$ and $M/K$. Hence the only upper jump is $i$.
 \end{proof}

\subsection{$p^2$-cyclic extensions}
 We will now consider $p^2$-cyclic extensions and their compositum. We begin by calculating the ramification filtration of a totally ramified $p^2$-cyclic extension of a local field $K$ of characteristic $p$ with algebraically closed residue field $k$. The valuation associated to $K$ will be denoted by $v_K$.

 \begin{lemma}\label{p^2-cyclic-filtration}
  Let $L/K$ be a $p^2$-cyclic totally ramified extension. Then there exist a Witt vector $(a_0,a_1)\in W_2(L)$ such that $(a_0^p,a_1^p)-_w(a_0,a_1)=(\alpha_0,\alpha_1)\in W_2(K)$ is reduced and $L=K(a_0,a_1)$. Let $n_i=-v_K(\alpha_i)$. The jumps in the lower ramification filtration of $L/K$ are $n_0, n_0(p^2-p+1)$ if $n_1\le n_0p$ otherwise it is $n_0, p(n_1-n_0)+n_0$. 
 \end{lemma}

 \begin{proof}
  The first statement follows from Artin-Schrier-Witt theory and Remark \ref{reduced}. So $n_0$ and $n_1$ are coprime to $p$. Let $x$ be a local parameter of $K$. Note that $a_0^p-a_0=\alpha_0=u_0x^{-n_0}$ for some unit $u_0$ in the valuation ring $R$ of $K$. Let $M=K(a_0)$ be the intermediate $p$-cyclic extension. Note that $v_M(x)=p$ and $v_M(a_0)=-n_0$. Let $a=a_0^{1/n_0}$, $y=a^{-1}$. Note that as the residue field is algebraically closed, by Hensel's lemma, $y\in M$. Using the formula \eqref{witt.substraction} for subtraction in the Witt ring, we note that 
  \begin{equation}\label{Witt.sub}
     a_1^p-a_1 =\alpha_1-a_0^{p^2-p+1}+\text{ terms of higher valuation} 
  \end{equation}
  If $n_1\le n_0(p-1)$ then $v_M(a_1^p-a_1)=-n_0(p^2-p+1)$. So the jumps in the lower ramification filtration are $n_0, n_0(p^2-p+1)$ because the lower ramification filtration behaves well with subgroups and the only jump in the lower ramification filtration of $M/K$ and $L/M$ are $n_0$ and $n_0(p^2-p+1)$ respectively.
 
  Now suppose that $n_1>n_0(p-1)$. In this case we shall modify $a_1$ to $\bar a_1$, so that $\bar a_1^p-\bar a_1\in M$ and $v_M(\bar a_1^p-\bar a_1)$ is prime to $p$. Note that
  \begin{align*}
   x^{n_0}&=u_0(a_0^p-a_0)^{-1}\\
   x^{n_0}&=u_0(a^{n_0p}-a^{n_0})^{-1}\\
   x^{n_0}&=y^{n_0p}(1-y^{pn_0-n_0})^{-1}u_0\\
   x&=y^p(1-y^{pn_0-n_0})^{-\frac{1}{n_0}}(u_0)^{-\frac{1}{n_0}}
  \end{align*}
  Using the above and noting that $\alpha_1=u_1x^{-n_1}$ for some unit $u_1\in K$, we can simplify equation \eqref{Witt.sub} to obtain
  \begin{equation}\label{eq1}
   a_1^p-a_1=uy^{-n_1p}(1-y^{pn_0-n_0})^{-\frac{n_1}{n_0}}-y^{-n_0p^2+n_0p-n_0}+\text{ higher valuation terms}
  \end{equation}
  Here $u\in K$ is a unit. We claim the following:
  \begin{clm}
   $uy^{-n_1p}(1-y^{pn_0-n_0})^{\frac{n_1}{n_0}}=(c_0y^{-n_1p}+c_1y^{-n_1p+p}++c_2y^{-n_1p+2p}+\ldots)+(b_0y^{(n_0-n_1)p-n_0}+\text{ higher valuation terms})$, where $c_i$ and $b_0$ are in the coefficient field $k$ with $c_0$ and $b_0$ nonzero.
  \end{clm}
  \begin{proof}[Proof of the claim]
   Since $v_M(y)=1$, by Cohen Structure theorem, $M\cong k((y))$. 
   It is clear from the left hand side of the equation that the leading term is $c_0y^{-n_1p}$ for some nonzero $c_0$ in the coefficient field $k$. So it is enough to show that the smallest $N$ such that the coefficient of $y^N$ is nonzero and $(N,p)=1$ is $(n_0-n_1)p-n_0$. Taking derivative with respect to $y$ of the left hand side, we get
   \begin{equation*}
    \frac{d}{dy}(LHS)=\frac{du}{dy}y^{-n_1p}(1-y^{pn_0-n_0})^{n_1/n_0}+n_1uy^{-n_1p}(1-y^{pn_0-n_0})^{\frac{n_1}{n_0}-1}y^{pn_0-n_0-1}   
   \end{equation*}
   Since $u\in K$ is a unit in $R$, $v_M(du/dy)\ge v_M(dx/dy)$. Moreover, $x=y^p(1-y^{pn_0-n_0})^{-\frac{1}{n_0}}u_0^{-\frac{1}{n_0}}$ implies that
   $$e_0x+e_1x^2+e_2x^3+\ldots=y^p(1-y^{pn_0-n_0})^{-\frac{1}{n_0}}$$
   for some $e_0, e_1, \ldots$ in $k$ with $e_0\ne 0$. Differentiating the above equation with respect to $y$ we get,
   $$\frac{dx}{dy}(e_0+2e_1x+3e_2x^2+\ldots)=-y^p(1-y^{pn_0-n_0})^{-\frac{1}{n_0}-1}y^{pn_0-n_0-1}$$
   So $v_M(\frac{dx}{dy})=pn_0-n_0+p-1$. So $v_M(\frac{d}{dy}(LHS))=(n_0-n_1)p-n_0-1$. This shows that $N=(n_0-n_1)p-n_0$.
  \end{proof}

  Let $\bar a_1=a_1-c_0^{1/p}y^{-n_1}-c_1^{1/p}y^{-n_1+1}-\ldots -c_{n_0}^{1/p}y^{-n_1+n_0}$. Then from \eqref{eq1} and the claim it follows that
  $$\bar a_1^p-\bar a_1=c_0^{1/p}y^{-n_1}+b_0y^{(n_0-n_1)p-n_0}-y^{-n_0p^2+n_0p-n_0}+\text{ terms of higher valuation}$$
  Also note that $\bar a_1^p-\bar a_1\in M$ and $M(\bar a_1)=M(a_1)=L$.
 
  Now if $n_0(p-1)<n_1\le n_0p$ then $v_M(\bar a_1^p-\bar a_1)=-n_0(p^2-p+1)$. So the jumps in the lower ramification filtration in this case as well are $n_0, n_0(p^2-p+1)$.
 
  Finally if $n_1 > n_0p$ then $p(n_0-n_1)-n_0 < -n_0(p^2-p+1)$ and $p(n_0-n_1)-n_0 < -n_1$. So $v_M(\bar a_1^p-\bar a_1)=p(n_0-n_1)-n_0$ and the jumps in the lower ramification filtration are $n_0, p(n_1-n_0)+n_0$.
 \end{proof}

 Let $L/K$ and $M/K$ be distinct totally ramified $p^2$-cyclic extension of $K$ corresponding to the reduced Witt vectors $(\alpha_0,\alpha_1)$ and $(\beta_0,\beta_1)$ in $W_2(K)$. We will compute the ramification filtration of the compositum $LM/K$ in some cases. 
 
 \begin{pro}\label{filtration.linearly.disjoint}
  Suppose $L$ and $M$ are linearly disjoint over $K$. Let the upper jumps of $L/K$ and $M/K$ be $u_0,u_1$ and $v_0,v_1$ respectively. Let the upper jumps of $G=\Gal(LM/K)$ be $w_0,w_1,\ldots$. The following holds:
  \begin{enumerate}
   \item \[ (u_0,u_1) = \begin{cases}
                   (-v_K(\alpha_0),-v_K(\alpha_0)(p-1)), & \text{if }v_K(\alpha_1)\ge v_K(\alpha_0)p \\
                   (-v_K(\alpha_0),v_K(\alpha_1)), & \text{otherwise.}
                       \end{cases}
        \] \[
           (v_0,v_1) = \begin{cases}
                   (-v_K(\beta_0),-v_K(\beta_0)(p-1)), & \text{if }v_K(\beta_1)\ge v_K(\beta_0)p \\
                   (-v_K(\beta_0),v_K(\beta_1)), & \text{otherwise.}
                       \end{cases}         
           \]
   \item  If $u_0\ne v_0$ then $w_0=\min(u_0,v_0)$. Moreover if $u_0,u_1,v_0,v_1$ are all distinct then these are the four upper jumps of $G$.
   \item Suppose $u_0=v_0$. If there exist $c\in \FF_p$ and $x\in K$ such that $l=-v_K(\alpha_0+c\beta_0+x^p-x)<u_0$ then $w_0=l$ otherwise $w_0=u_0$. Moreover, if $u_1\ne v_1$ then $l,u_0, u_1, v_1$ are the only upper jumps of $G$ in the first case and $w_0=u_0,u_1,v_1$ are the only upper jumps of $G$ with $G^{w_0+}=(\Z/p\Z)^2$ in the latter case.
  \end{enumerate}
 \end{pro}

 \begin{proof}
  Statement (1) is a direct consequence of Lemma \ref{p^2-cyclic-filtration} and Remark \ref{upper-lower-conversion}. Without loss of generality, we may assume $u_0\le v_0$.
 
  Note that $G=\Gal(L/K)\times \Gal(M/K)$. Let $H_1$ and $H_2$ be the $p$-cyclic subgroups of $\Gal(L/K)$ and $\Gal(M/K)$ respectively. Let $(a_0,a_1)\in W_2(L)$ and $(b_0,b_1)\in W_2(M)$ be such that $(a_0^p,a_1^p)-_w(a_0,a_1)=(\alpha_0,\alpha_1)$ and $(b_0^p,b_1^p)-_w(b_0,b_1)=(\beta_0,\beta_1)$. Note that $L^{H_1}=K(a_0)$ and $M^{H_2}=K(b_0)$. 
 
  If $u_0\ne v_0$ then the upper jumps of $\Gal(K(a_0,b_0)/K)=G/(H_1\times H_2)$ are $u_0,v_0$ by Proposition \ref{p-cyclic-compositum}. By Lemma \ref{upperjumpslift}, $u_0$ is an upper jump of $G$ which implies $w_0\le u_0$. Also since $u_0$ is the first upper jump of $G/(H_1\times H_2)$ and $(G/(H_1\times H_2))^{u_0+}\ne 1$, we have 
  $$G^{u_0}(H_1\times H_2)=G \text{ and }$$ 
  $$H_1\times H_2 \subsetneq G^{u_0+}(H_1\times H_2) \subsetneq G.$$ 
  Since $w_0$ is the first upper jump of $G$, by \cite[IV, 2, Corollary 3]{Serre-local.fields} $G/G^{w_0+}$ is a group of exponent $p$. But $G=(\Z/p^2\Z)^2$ so $G^{w_0+}\supset H_1\times H_2$. If $w_0< u_0$ then $G^{w_0+}\supset G^{u_0}$. But this implies $G^{w_0+}=G$ which contradicts that $w_0$ is an upper jump of $G$. So $w_0=u_0$. The moreover part of the statement (2) follows from Corollary \ref{cor.to.upperjumpslift}.
 
  For statement (3), we note that by Proposition \ref{p-cyclic-compositum}, the upper jumps of $G/(H_1\times H_2)=\Gal(K(a_0,b_0)/K)$ are $l, u_0$ if there exist $c\in \FF_p$ and $x\in K$ such that $l=-v_K(\alpha_0+c\beta_0+x^p-x)<u_0$. So by Lemma \ref{upperjumpslift}, $l, u_0, u_1 ,v_1$ are upper jumps of $G$. Also since the upper jumps of $G/(H_1\times H_2)$ are $l, u_0$, we are in the previous setup. Hence $w_0=l$ and moreover if $u_1\ne v_1$ then $l, u_0, u_1,v_1$ are all distinct. So by Corollary \ref{cor.to.upperjumpslift} these are all the upper jumps of $G$.
 
  In the case where no such $c$ and $x$ exist, again by Proposition \ref{p-cyclic-compositum}, the only upper jump of $G/(H_1\times H_2)$ is $u_0$. So $(G/(H_1\times H_2))^{u_0+}=1$ and $(G/(H_1\times H_2))^{u_0}=G/H_1\times H_2$. But this is equivalent to $G^{u_0+}\subset H_1\times H_2$ and $G^{u_0}(H_1\times H_2)=G$. Again $G^{w_0+}\supset H_1\times H_2$ and if $w_0< u_0$ then $G^{w_0+} \supset G^{u_0}$ which would imply $G^{w_0+}=G$ contradicting that $w_0$ is an upper jump. Hence $w_0=u_0$.
  Also $G^{w_0+}$ has index at most $p^2$ in $G$. Moreover $u_1$ and $v_1$ are both greater than $w_0$ and they are upper jumps of $G$. Hence $G^{w_0+}=H_1\times H_2$ is exactly of index $p^2$ and $u_0, u_1, v_1$ are the only upper jumps of $G$. 
 \end{proof}

 \begin{pro}\label{filtration.not.linearly.disjoint}
  Suppose $L$ and $M$ are not linearly disjoint over $K$. Then $L\cap M$ is a $p$-cyclic extension of $K$ and $G=\Gal(LM/K) \cong \Z/p^2\Z\times \Z/p\Z$. Let $u_0,u_1$ and $v_0,v_1$ be the upper jumps of $L/K$ and $M/K$ respectively. The following holds:
  \begin{enumerate}
   \item $L\cap M=K(a_0)=K(b_0)$ and $u_0=v_0$ is the upper jump of $\Gal(L\cap M/K)$.
   \item If $u_1\ne v_1$ then $u_0,u_1, v_1$ are the upper jumps of $G$.
   \item If $u_1=v_1$, suppose $\alpha_0=c\beta_0$ and $v_K(\alpha_1-c\beta_1)$ is different from $-u_0$ and $-u_1$ for some nonzero $c\in \FF_p$, then the upper jumps of $G$ are $u_0,u_1=v_1, -v_K(\alpha_1-c\beta_1)$.
  \end{enumerate}
 \end{pro}

 \begin{proof}
  Note that $G=\Gal(LM/K)=\Gal(L/K)\times_{\Gal(L\cap M/K} \Gal(M/K)\cong \Z/p^2\Z\times \Z/p\Z$.
  Since $L$ and $M$ are distinct and not linearly disjoint, $L\cap M/K$ is a $p$-cyclic extension. Note that $K(a_0)$ and $K(b_0)$ are the unique $p$-cyclic subextensions of $K$ contained in $L$ and $M$ respectively. So $L\cap M=K(a_0)=K(b_0)$ and $u_0=v_0$. Moreover the only upper jump of $K(a_0)/K$ is $u_0$. This proves (1).
 
  For (2), we note that $u_0, u_1, v_1$ are upper jumps of $G$ by Lemma \ref{upperjumpslift}. Since $u_1\ne v_1$, $u_0, u_1, v_1$ are all distinct. So by Corollary \ref{cor.to.upperjumpslift}, $u_0, u_1, v_1$ are the only upper jumps of $G$.
 
  Finally for (3), let $(a_0^p,a_1^p)-_w(a_0,a_1)=(\alpha_0,\alpha_1)$ and $(b_0^p,b_1^p)-_w(b_0,b_1)=(\beta_0,\beta_1)$. Since $\alpha_0=c\beta_0$ for some $c\in \FF_p$, we obtain that $a_0=cb_0+d$ for some $d\in \FF_p$. Since $(d_0^p,d_1^p)-_w(d_0,d_1)=(0,0)$ for $d_0,d_1\in \FF_p$, we can modify $(b_0,b_1)$ if necessary to further assume that $d=0$, i.e., $a_0=cb_0$. By the subtraction formula \ref{witt.substraction} of Witt vectors, we note that
  \begin{align*}
   a_1^p-a_1&= \alpha_1+[-a_0^{p(p-1)}a_0+\frac{(p-1)}{2}a_0^{p(p-2)}a_0^2-\ldots+\frac{(p-1)(p-2)\ldots 2}{2\cdot 3\ldots (p-1)}a_0^pa_0^{p-1}] \\
           &= \alpha_1+c[-b_0^{p(p-1)}b_0+\frac{(p-1)}{2}b_0^{p(p-2)}b_0^2-\ldots+\frac{(p-1)(p-2)\ldots 2}{2\cdot 3\ldots (p-1)}b_0^pb_0^{p-1}]\\
           &=\alpha_1-c\beta_1+c(\beta_1 + [-b_0^{p(p-1)}b_0+\ldots+\frac{(p-1)(p-2)\ldots 2}{2\cdot 3\ldots (p-1)}b_0^pb_0^{p-1}])\\
           &=\alpha_1-c\beta_1+c(b_1^p-b_1)
  \end{align*}

  Hence we obtain that $(a_1- cb_1)^p-(a_1-cb_1)=\alpha_1-c\beta_1$. So $K(a_1-cb_1)/K$ is a $p$-cyclic extension of $K$. Therefore by Lemma \ref{AS-upper}, $-v_K(\alpha_1-c\beta_1)$ is an upper jump of $G$. The hypothesis that $-v_K(\alpha_1-c\beta_1)$ is different from $u_0$ and $u_1$ implies that the only upper jumps of $G$ are $u_0, u_1$ and $-v_K(\alpha_1-c\beta_1)$ (Corollary \ref{cor.to.upperjumpslift}).
 \end{proof}

\section{Global applications}

 For a $p$-cyclic cover $X\to \PP^1_x$ branched at one point, the genus of $X$ is easy to compute and has a well known formula. Any such cover is given by an Artin-Schrier extension $Z^p-Z-f$ where $f\in k[x]$ is a polynomial of degree $r$ coprime to $p$. The genus of $X$ is $(p-1)(r-1)/2$. This is calculated using Hilbert's different formula, Riemann-Hurwitz formula and explicitly determining the higher ramification filtration at the totally ramified point of $X$ above $\infty$. 

 We can now carry out this computation for Galois covers of $\PP^1$ of degree $p^2$ branched only at one point with the aid of Proposition \ref{p-cyclic-compositum} and Lemma \ref{p^2-cyclic-filtration}. From now on the base field $k$ is also assumed to be algebraically closed.

 \begin{thm}\label{genus-p^2-noncyclic-covers}
  Let $\Phi:X\to \PP^1_x$ be a $(\Z/p\Z)^2$-cover of $\PP^1$ branched only at $\infty$. At the point $x=\infty$, after passing to completion, this cover induces a $(\Z/p\Z)^2$-Galois extension of local fields $L/k((x^{-1}))$ where $L=\QF(\hat \cO_{X,\Phi^{-1}(\infty)})$. Let $\alpha_0, \alpha_1\in k((x^{-1}))$ be such that $L$ is the compositum of distinct Artin-Schrier extensions $Z^p-Z-\alpha_0$ and $Z^p-Z-\alpha_1$ and $n_i=-v(\alpha_i)$ are coprime to $p$ for $i=0,1$. Here $v$ is the valuation of of the local field $k((x^{-1}))$. Then the genus of $X$ is $\frac{1}{2}[(n_1-1)p^2-(n_1-n_0)p -n_0 +1]$ if $n_0 < n_1$. If $n_0=n_1$ and $-v(\alpha_0+a\alpha_1+y^p-y)\ge n_0$ for all $a\in \FF_p$ and $y\in k((x^{-1}))$ then the genus is $\frac{1}{2}[(n_0-1)p^2 -n_0 +1]$. If $n_0=n_1$ and for some $a\in \FF_p$ and $y\in k((x^{-1}))$, $n'_0=-v(\alpha_0+a\alpha_1+y^p-y)< n_0$ then the genus is $\frac{1}{2}[(n'_0-1)p^2-(n_0-n'_0)p -n'_0 +1]$.
 \end{thm}

 \begin{proof}
  By Riemann-Hurwitz formula the genus of $X$ is $\frac{1}{2}[p^2(-2)+\degree(R)+2]$, where $R$ is the ramification divisor. By Hilbert's different formula and the fact that there is only one ramified point in $X$, $\degree(R)=\Sigma_{i\ge 0} (|I_i|-1)$, where $I_i$ are the lower filtration of the inertia group at the ramified point of $X$. 
  
  Proposition \ref{p-cyclic-compositum} gives the upper ramification filtration on the inertia group. Using Remark \ref{upper-lower-conversion}, we obtain the lower ramification filtration. The result follows from a straight forward arithmetic calculation. Let us perform this calculation in the case $n_0<n_1$.
  The upper jumps in this case are $n_0$ and $n_1$. So the lower jumps are $n_0$ and $(n_1-n_0)p+n_0$ by Remark \ref{upper-lower-conversion}. So by Hilbert's different formula $\degree(R)=(n_0+1)(p^2-1)+(n_1-n_0)p(p-1)$. So the genus of $X$ is $\frac{1}{2}[-2p^2+n_0p^2+p^2-n_0-1+n_1p^2-n_1p-n_0p^2+n_0p+2]$. Simplify the expression to obtain the result.
 \end{proof}

 \begin{thm}\label{genus-p^2-cyclic-covers}
  Let $\Phi:X\to \PP^1_x$ be a $(\Z/p^2\Z)$-cover of $\PP^1$ branched only at $\infty$. At the point $x=\infty$, after passing to completion, this cover induces a $(\Z/p\Z)^2$-Galois extension of local fields $L/k((x^{-1}))$ where $L=\QF(\hat \cO_{X,\Phi^{-1}(\infty)})$. Let $\alpha_0, \alpha_1\in k((x^{-1}))$ be such that $L$ is the Artin-Schrier-Witt extension corresponding to the Witt vector $(\alpha_0,\alpha_1)$ and $n_i=-v(\alpha_i)$ are coprime to $p$ for $i=1,2$. Here $v$ is the valuation of the local field $k((x^{-1}))$. Then the genus of $X$ is $\frac{1}{2}[n_0(p-1)(p^2+1)-p^2+1]$ if $n_1 \le pn_0$ and $\frac{1}{2}[(n_1-1)p^2-(n_1-n_0)p -n_0 +1]$ otherwise.
 \end{thm}
 \begin{proof}
  The proof is the same as above. The only difference is that Proposition \ref{p^2-cyclic-filtration} is used instead of Proposition \ref{p-cyclic-compositum}. 
 \end{proof}

\subsection{Applications to Abhyankar's Inertia conjecture}
 Let $G$ be quasi-$p$ group and $I\le G$ be a subgroup. Recall that we say $(G,I)$ is realizable if there exist a $G$-Galois cover $X\to \PP^1$ branched only at one point $\infty$ and the inertia group at a point of $X$ above $\infty$ is $I$.

 \begin{thm}\label{AIC.known}
  Suppose $(G,I)$ is realizable and let $P$ be a $p$-group then 
  \begin{enumerate}
   \item $(G\times P, I\times P)$ is realizable.
   \item $(G,Q)$ is realizable where $Q$ is any $p$-subgroup of $G$ containing $I_2$ if there is no epimorphism from $G$ to any nontrivial quotient of $I_2$.
  \end{enumerate}
 \end{thm}
 
 \begin{proof}
  Since $(G,I)$ is realizable, there exist a $G$-cover $X\to \PP^1$ branched only at $\infty$ and the inertia group at a point $r\in X$ above $\infty$ is the subgroup $I$. This implies that the Galois group of the field extension $\QF(\hat \cO_{X,r})/\QF(\hat \cO_{\PP^1,\infty})$ is $I$. Since there are infinitely many linearly disjoint $P$-extension of $\QF(\hat \cO_{\PP^1,\infty})$ \cite{Harbater-moduli.of.pcovers}, there exist a $P$-extension $\hat L/\QF(\hat \cO_{\PP^1,\infty})$ linearly disjoint from $\QF(\hat \cO_{X,r})/\QF(\hat \cO_{\PP^1,\infty})$. Let $Y\to \PP^1$ be the Harbater-Katz-Gabber $P$-cover $Y\to \PP^1$ associated to the $P$-extension $\hat L/\QF(\hat \cO_{\PP^1,\infty})$. Note that $Y\to \PP^1$ is linearly disjoint to the cover $X\to \PP^1$. Letting $U$ to be the normalization of $X\times_{\PP^1} Y$ we note that $U \to \PP^1$ is a $G\times P$ cover branched only at $\infty$. Moreover, the linear disjointness of $\hat L$ and $\QF(\hat \cO_{X,r})$ over $\QF(\hat \cO_{\PP^1,\infty})$ implies that $\Gal(\hat L\QF(\hat \cO_{X,r})/\QF(\hat \cO_{\PP^1,\infty}))=I\times P$. By \cite[Lemma 3.1]{wild.ram}, the inertia group at the point $(r,\infty)$ of the cover $U\to \PP^1$ is $I\times P$.
 
  The second statement is a consequence of \cite[Theorem 3.7]{wild.ram} and \cite[Theorem 2]{Harbater-adding.branch.points}.
 \end{proof}

 \begin{thm}\label{p-cyclic.main}
  Suppose there is no epimorphism from $G\to \Z/p\Z$. If $(G,P)$ is realizable where $P\le G$ is a $p$-group then there exist an index $p$-subgroup $Q$ of $P \times \Z/p\Z$ such that $(G\times \Z/p\Z, Q)$ is realizable. 
 \end{thm}

 \begin{proof}
  Since $(G,P)$ is realizable, there exist a $G$-cover $X\to \PP^1_x$ branched only at $x=\infty$ and the inertia group at a point $r\in X$ above $\infty$ is the subgroup $P$. So the Galois group of the field extension $\QF(\hat \cO_{X,r})/\QF(\hat \cO_{\PP^1,\infty})$ is $P$. Let $L=\QF(\hat \cO_{X,r})$ and $K=\QF(\hat \cO_{\PP^1,\infty})=k((x^{-1}))$.
 
  Let $P_1$ be an index $p$ (normal) subgroup of $P$. Then $L^{P_1}$ is an Artin-Schrier extension of $K$. Let $\alpha\in L$ be a reduced AS-element such that $L^{P_1}=K(\alpha)$ and $\beta=\alpha^p-\alpha$. Note that $\beta=c_nx^n+c_{n-1}x^{n-1}+\ldots c_1x+c_0 +c_{-1}x^{-1}+\ldots$ for some $n$ coprime to $p$, $c_n,c_{n-1},\ldots \in k$ and $c_n\ne 0$. Let $\beta'=\beta+cx$ for some nonzero $c\in k$ such that $Z^p-Z-\beta'$ is an irreducible polynomial in $L[Z]$. Let $\alpha'\in \bar L$ be such that $\alpha'^p-\alpha'=\beta'$. Then $L$ and $K(\alpha')$ are linearly disjoint over $K$. Let $Y\to \PP^1$ be the $p$-cyclic Harbater-Katz-Gabber cover associated to the extension of local fields $K(\alpha')/K$ and $\infty_Y$ be the point lying above $x=\infty$. Note that the covers $Y\to \PP^1_x$ and $X\to \PP^1_x$ are linearly disjoint because $k(Y)K=K(\alpha')$ and $k(X)K=L$.
  
  Let $U$ be the normalization of $X\times_{\PP^1_x} Y$. Then $U$ is smooth and irreducible. The cover $U\to \PP^1_x$ is a $G\times \Z/p\Z$ cover branched only at $x=\infty$. Moreover by \cite[Lemma 3.1]{wild.ram} the inertia group and the ramification filtration at the point $r'=(r, \infty_Y)\in U$ is given by the extension of local fields $L(\alpha')/K$. So the inertia group is $I=P\times \Z/p\Z$. Set $\gamma=\alpha'-\alpha\in L(\alpha')$. Then $\gamma^p-\gamma=cx$. Since $K(\gamma)\subset L(\alpha')$, there is an induced epimorphism on the Galois groups $\phi:I\to \Z/p\Z$. Let $Q=ker(\phi)$ then $Q$ is an index $p$ subgroup of $P\times \Z/p\Z$ and $L(\alpha')^Q=K(\gamma)$.
  
  We note that $k(x)(\gamma)$ and $k(U)$ are linearly disjoint over $k(x)$. To prove this, let us assume the contrary. Then $k(x)(\gamma)\subset k(U)=k(X)k(Y)$ which induces an epimorphism on Galois groups $\phi:G\times\Z/p\Z\to \Z/p\Z$. By assumption $\phi|_G$ is not surjective and hence trivial. Hence $G=ker(\phi)$ and by Galois theory $(k(X)k(Y))^G=k(x)(\gamma)$. But this implies $k(Y)=k(x)(\gamma)$ and hence $K(\alpha')=K(\gamma)$, a contradiction.
  
  Let $V\to \PP^1_x$ be the $p$-cyclic cover corresponding to the extension $k(x)(\gamma)/k(x)$ and $W$ be the normalization of $U\times_{\PP^1_x} V$. Let $\infty_V\in V$ be the point lying above $x=\infty$ and $r''=(r',\infty_V)$. By \cite[Proposition 3.5]{wild.ram}, the inertia group of the cover $W\to V$ at $r''$ is $Q$. Since $k(U)$ and $k(V)$ are linearly disjoint over $k(x)$, we get that $W$ is connected and $\Gal(k(W)/k(V))=\Gal(k(U)/k(x))=G\times\Z/p\Z$. Moreover, $W\to V$ is branched only at $\infty_V$.
  Finally, since $V$ is isomorphic to $\PP^1$, we get that $(G\times \Z/p\Z, Q)$ is realizable. 
 \end{proof}

 \begin{cor}
  Suppose $(G,I)$ is realizable and $P\le G$ be any $p$-group containing $I_2$. Also assume that there is no epimorphism from $G$ to $\Z/p\Z$. Then there exist an index $p$-subgroup $Q$ of $P \times \Z/p\Z$ such that $(G\times \Z/p\Z, Q)$ is realizable. 
 \end{cor}
 \begin{proof}
  It follows from Theorem \ref{AIC.known} and Theorem \ref{p-cyclic.main}.
 \end{proof}

 \begin{thm}\label{p^2-cyclic.main}
  Suppose $(G,P)$ is realizable, there is no epimorphism from $G\to \Z/p\Z$ and $P$ has a $p^2$-cyclic quotient $a:P\to \Z/p^2\Z$. Then there exist an index $p$ subgroup $Q$ of $P\times_{\Z/p\Z} \Z/p^2\Z$ such that $(G\times \Z/p^2\Z, Q)$ is realizable.
 \end{thm}

 \begin{proof}
  As in the previous proof, there exist a $G$-cover $X\to \PP^1$ branched only at $\infty$ and the inertia group at a point $r\in X$ above $\infty$ is the subgroup $P$. The Galois group of the field extension $L/K$ is $P$ where $L=\QF(\hat \cO_{X,r})$ and $K=\QF(\hat \cO_{\PP^1,\infty})=k((x^{-1}))$. Let $P_1=\ker(a)$. Then $L^{P_1}$ is an Artin-Schrier-Witt extension of $K$ corresponding to a reduced Witt vector of length two. Let $(\alpha_0,\alpha_1)\in W_2(L)$ be such that $(\alpha_0^p,\alpha_1^p)-(\alpha_0,\alpha_1)=(\beta_0,\beta_1)\in W_2(K)$ and $L^{P_1}=K(\alpha_0,\alpha_1)$. 
  Let $\beta_1'=\beta_1+cx$ for some nonzero $c\in k$ such that the Artin-Schrier-Witt extension of $K$ corresponding to the Witt vector $(\beta_0,\beta_1')$ is a $\Z/p^2\Z$ extension of $K$ different from $K(\alpha_0,\alpha_1)$. Let $\alpha_1'$ be such that $(\alpha_0^p,\alpha_1'^p)-(\alpha_0,\alpha_1')=(\beta_0,\beta_1')$ then $K(\alpha_0,\alpha_1)$ and $K(\alpha_0,\alpha_1')$ are linearly disjoint over $K(\alpha_0)$. Let $Y\to \PP^1$ be the Harbater-Katz-Gabber cover associated to the local fields extension $K(\alpha_0,\alpha_1')/K$. So $k(Y)K=K(\alpha_0,\alpha_1')$.
 
  Since there is no epimorphism from $G\to \Z/p\Z$, the extensions $k(Y)/k(x)$ and $k(X)/k(x)$ are linearly disjoint. To see this, let $M=k(X)\cap k(Y)$ and suppose $k(x)\subsetneq M$. Then $M/k(x)$ is either $p$-cyclic or $p^2$-cyclic. In either case, we get an epimorphism from $G\to \Z/p\Z$ contradicting the hypothesis.
 
  Letting $U$ to be the normalization of $X\times_{\PP^1_x} Y$ we note that $U$ is smooth and irreducible. The cover $U\to \PP^1_x$ is a $G\times \Z/p^2\Z$ cover branched only at $\infty$ and the inertia group and the ramification filtration at $r'=(r,\infty_Y)\in U$ is given by the extension of local fields $LK(\alpha_0,\alpha_1')/K$. Hence the inertia group is $I=P\times_{\Z/p\Z} \Z/p^2\Z$. By Proposition \ref{filtration.not.linearly.disjoint}(3) $-v_K(\beta_1-\beta_1')=1$ is an upper jump of $I$. Moreover, letting $\alpha=\alpha_1-\alpha'_1$, we note that $\alpha^p-\alpha=\beta_1-\beta'_1\in K$ and $\alpha\notin K$. Let $Q$ be the index $p$ subgroup of $I$ given by $\Gal(K(\alpha_0,\alpha_1,\alpha'_1)/K(\alpha))$. Note that $K(\alpha)/K$ is a $p$-cyclic extension with the lower jump at $1$. 
  So the Harbater-Katz-Gabber cover $V\to \PP^1_x$ associated to this extension has the property that $V$ is isomorphic to $\PP^1$. As in the above proof $k(U)$ and $k(V)$ are linearly disjoint over $k(x)$. Let $W$ be the normalization of $U\times_{\PP^1_x} V$ then $W$ is smooth and irreducible. Again we apply \cite[Proposition 3.5]{wild.ram} to conclude that $(G\times \Z/p^2\Z, Q)$ is realizable.
 \end{proof}

 \begin{cor}
  Suppose $(G,I)$ is realizable. Let $P$ be any $p$-subgroup containing $I_2$ and assume that $P$ has a $p^2$-cyclic quotient $a:P\to \Z/p^2\Z$. Also assume that there is no epimorphism from $G$ to $\Z/p\Z$. Then there exist an index $p$ subgroup $Q$ of $P\times_{\Z/p\Z} \Z/p^2\Z$ such that $(G\times \Z/p^2\Z, Q)$ is realizable.
 \end{cor}

\end{document}